\documentclass[12pt]{article}
\usepackage[utf8]{inputenc}
\usepackage{graphicx}
\usepackage{amsmath}
\usepackage{fullpage}
\usepackage{mathtools}
\usepackage{commath}
\usepackage{csquotes}
\usepackage{amsfonts}
\usepackage{amssymb}
\usepackage{amsthm}
\newtheorem{thm}{Theorem}[section]
\newtheorem{lem}{Lemma}[section]

\newtheorem{dfn}{Definition}[section]

\newtheorem{prob}{Problem}[section]

\DeclareMathOperator{\MAX}{MAX}
\title{Eccentricity energy change of complete multipartite graphs due to edge deletion}
\author{Iswar Mahato \thanks{Department of Mathematics, Indian Institute of Technology Kharagpur, Kharagpur 721302, India. Email: iswarmahato02@gmail.com}\  \and M. Rajesh Kannan\thanks{Department of Mathematics, Indian Institute of Technology Kharagpur, Kharagpur 721302, India. Email: rajeshkannan@maths.iitkgp.ac.in, rajeshkannan1.m@gmail.com }}
\date{\today}
\begin{document}
\maketitle

\begin{abstract}
The eccentricity matrix $\varepsilon(G)$ of a graph $G$ is obtained from the distance matrix of $G$ by retaining the largest distances in each row and each column, and leaving zeros in the remaining ones. The eccentricity energy of $G$ is sum of the absolute values of the eigenvalues of $\varepsilon(G)$. Although the eccentricity matrices of graphs are closely related to the distance matrices of graphs, a number of properties of eccentricity matrices are substantially different from those of the distance matrices. The change in eccentricity energy of a graph due to an edge deletion is one such property. In this article, we give examples of graphs for which the eccentricity energy increase (resp., decrease) but the distance energy decrease (resp., increase) due to an edge deletion. Also, we prove that the eccentricity energy of the complete $k$-partite graph $K_{n_1,\hdots,n_k}$ with $k\geq 2$ and $ n_i\geq 2$,  increases due to an edge deletion.
\end{abstract}



\section{Introduction}\label{sec1}
Throughout this paper,  we consider finite simple connected graphs. For a graph $G$,  let $V(G)$ and $E(G)$ denote the vertex set and the edge set of $G$, respectively. \textit{The adjacency matrix }of a graph $G$ on $n$ vertices, denoted by $A(G)$, is an $n \times n$ matrix whose rows and columns are indexed by the vertices of $G$ and the entries are defined as:  $A(G)=(a_{uv})$, where $a_{uv}=1$ if the vertices $u$ and $v$ are adjacent and $a_{uv}=0$ otherwise. The set of all eigenvalues of $A(G)$ is the spectrum of $G$. \textit{The energy (or the adjacency energy) of $G$} is defined as
$E_A(G)=\sum_{i=1}^n |\lambda_i|,$
where $\lambda_1,\lambda_2,\hdots,\lambda_n$ are the eigenvalues of $A(G)$. The \textit{distance} between the vertices $u$ and $v$ in $G$, denoted by  $d_G(u,v)$,  is the length of a shortest path between them.  The \textit{distance matrix} of a connected graph $G$ on $n$ vertices, denoted by $D(G)$, is an $n\times n$ matrix whose rows and columns are indexed by the vertices of $G$ and the entries are defined by $D(G)_{uv}=d_G(u,v)$. The eigenvalues of $D(G)$ are \emph{the distance eigenvalues} of $G$ and the largest eigenvalue of $D(G)$ is  \textit{the distance spectral radius} of $G$. The distance energy of $G$ is defined by
$E_D(G)=\sum_{i=1}^n |\mu_i|,$
where $\mu_1,\mu_2,\hdots,\mu_n$ are the distance eigenvalues of $G$.

The \textit{eccentricity} of a vertex $u\in V(G)$, denoted by  $e(u)$,  is defined by $e(u)=\max\{d_G(u,v):v\in V(G)\}$. \textit{The eccentricity matrix} $\varepsilon(G)$ of a graph $G$ is an $n\times n$ matrix indexed by the vertices of $G$ and the entries are defined as
$${\varepsilon(G)}_{uv}=
\begin{cases}
\text{$d_G(u,v)$} & \quad\text{if $d_G(u,v)=\min\{e(u),e(v)\}$,}\\
\text{0} & \quad\text{otherwise.}
\end{cases}$$
In \cite{ran1},  Randi\'{c}  introduced the notion of eccentricity matrix of a graph, then known as $D_{\MAX}$-matrix. It was renamed as eccentricity matrix by Wang et al. in \cite{ecc-main}. The eigenvalues of $\varepsilon(G)$ are \textit{the $\varepsilon$-eigenvalues} of $G$ and the set of all $\varepsilon$-eigenvalues of $G$ is the $\varepsilon$-spectrum of $G$. The largest eigenvalue of $\varepsilon(G)$ is   \textit{the $\varepsilon$-spectral radius} and is denoted by $\rho_{\varepsilon}(G)$. \textit{The eccentricity energy (or the $\varepsilon$-energy) }of a graph $G$ is defined \cite{wang2019graph} as
$E_{\varepsilon}(G)=\sum_{i=1}^n |\xi_i|,$
where $\xi_1,\xi_2,\hdots,\xi_n$ are the eigenvalues of $\varepsilon(G)$.

Recently, the eccentricity matrices stimulated a lot of interest and brought the attention of researchers. By looking at the definition, one may lead to think that the eccentricity matrices of graphs may behave like the distance matrices. But this is not true in general. For example, the eccentricity matrices of connected graphs could be reducible, while the distance matrices are always irreducible. The eccentricity matrix of a complete bipartite graph on $n$ vertices with maximum degree less than $n-1$ is reducible; while the eccentricity matrices of trees of order $n\geq 2$ are irreducible \cite{ecc-main}. Indeed, characterizing the classes of graphs with irreducible eccentricity matrices is an interesting and nontrivial problem considered in the literature \cite{ecc-main,wang2020spectral}. The distance matrices of trees are always invertible, but the eccentricity matrices of trees need not be invertible. Recently, Mahato et al., \cite{mahato2019spectral} characterized all the trees with invertible eccentricity matrices. For  recent works on the eccentricity matrix, we refer to \cite{mahato2020,wang2019graph,ecc-main,wang2020spectral}.

In recent years, the problem of energy change (for various matrices associated with graphs) due to an edge deletion considered by the researchers. The adjacency energy change of graphs due to edge deletion have been studied in \cite{akbari2009edge,day2008graph,shan2017energy,wang2015graph}. Recently, in  \cite{varghese2018distance},  Varghese et al.,  verified that the distance energy of complete bipartite graphs increase when an edge is removed and conjectured that the complete multipartite graphs have the same property. In  \cite{tian2020}, Tian et al., proved that the conjecture holds for the complete multipartite graph $K_{p,\dots, p}$ and the tripartite Turán graph. In \cite{sun611proof, sun611proof-cor}, Sun and Das solved the conjecture affirmatively. Motivated by the above-mentioned works, we study the problem that change in the eccentricity energy of graphs due to an edge removal.  Note that the eccentricity energy of a graph may increase, decrease or remain the same due to an edge deletion (See Table \ref{tab1}).  

\begin{figure}[h!]
	\centering
	\includegraphics[scale= 0.40]{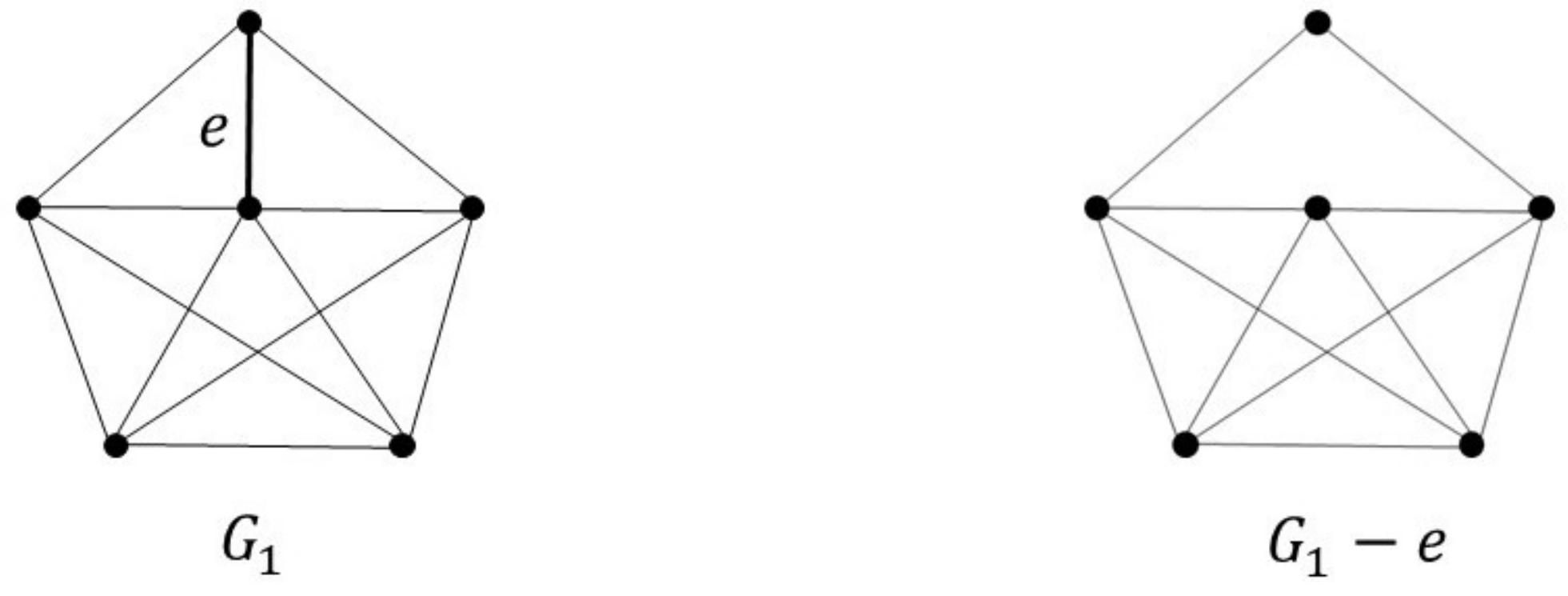}
	\caption{$G_1$ and $G_1-e$}
	\includegraphics[scale= 0.40]{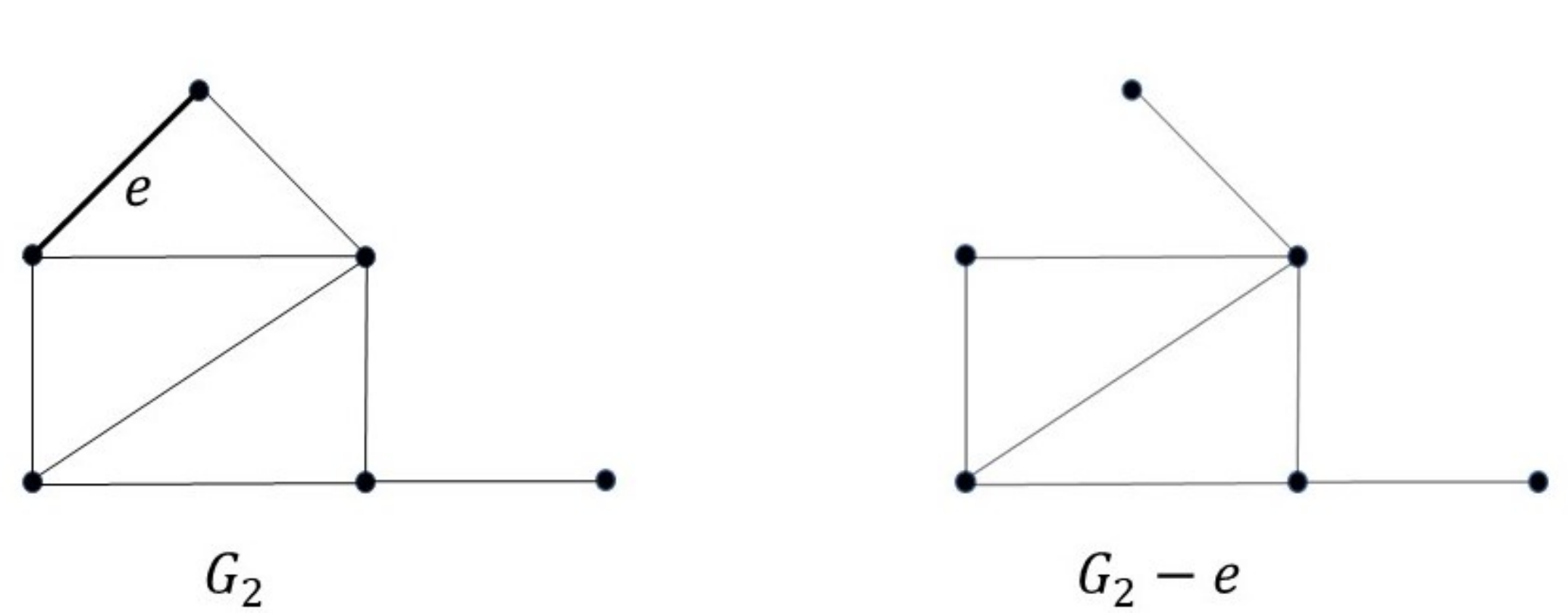}
	\caption{$G_2$ and $G_2-e$}
	\label{fig2}
\end{figure}

\begin{table}[h!]
	\centering
	\begin{tabular}{|c| c| c| c| c|}
		\hline
		Graphs & Eccentricity energy & Graphs & Eccentricity energy & Comparison\\
		\hline
		$K_6$ & = 10 & $K_6-e$ & $\approx 10.7446$ & $E_{\varepsilon}(K_6)< E_{\varepsilon}(K_6-e)$\\
		\hline
		$G_1$ & $\approx 12.2814$ & $G_1-e$ & $\approx 10.9282$ & $E_{\varepsilon}(G_1)> E_{\varepsilon}(G_1-e)$ \\
		\hline
		$G_2$ & $\approx 16.8327$ & $G_2-e$ & $\approx 16.8327$ & $E_{\varepsilon}(G_2)= E_{\varepsilon}(G_2-e)$\\
		\hline
	\end{tabular}
	\caption{Comparison of eccentricity energy change }
	\label{tab1}
\end{table}

The change of eccentricity energy of a graph due to an edge deletion is not similar to that of the distance energy change. The eccentricity energy of $G_3$ (resp., $G_4$) increases (resp., decreases), but the distance energy of $G_3$ decreases (resp., increases) when the edge $e$ is deleted (See, Table \ref{tab2}). The deletion of an edge from a connected graph with a unique positive distance eigenvalue always increases the distance energy provided that the resulting graph is  connected \cite{zhou2011distance}. But the deletion of an edge from a connected graph with exactly one positive $\varepsilon$-eigenvalue may not always increases the $\varepsilon$-energy. The graph $G_4$ in Figure \ref{fig4} has exactly one positive $\varepsilon$-eigenvalue, but $E_{\varepsilon}(G_4)=17.3808>16.9706=E_{\varepsilon}(G_4-e)$. This makes the problem more interesting. So we consider the following problem for the eccentricity energy of graphs:
\begin{prob}\label{prob1}
	Classify the graphs for which the eccentricity energy always increases or decreases due to an edge deletion.
\end{prob}
\begin{figure}[h!]
	\centering
	\includegraphics[scale= 0.40]{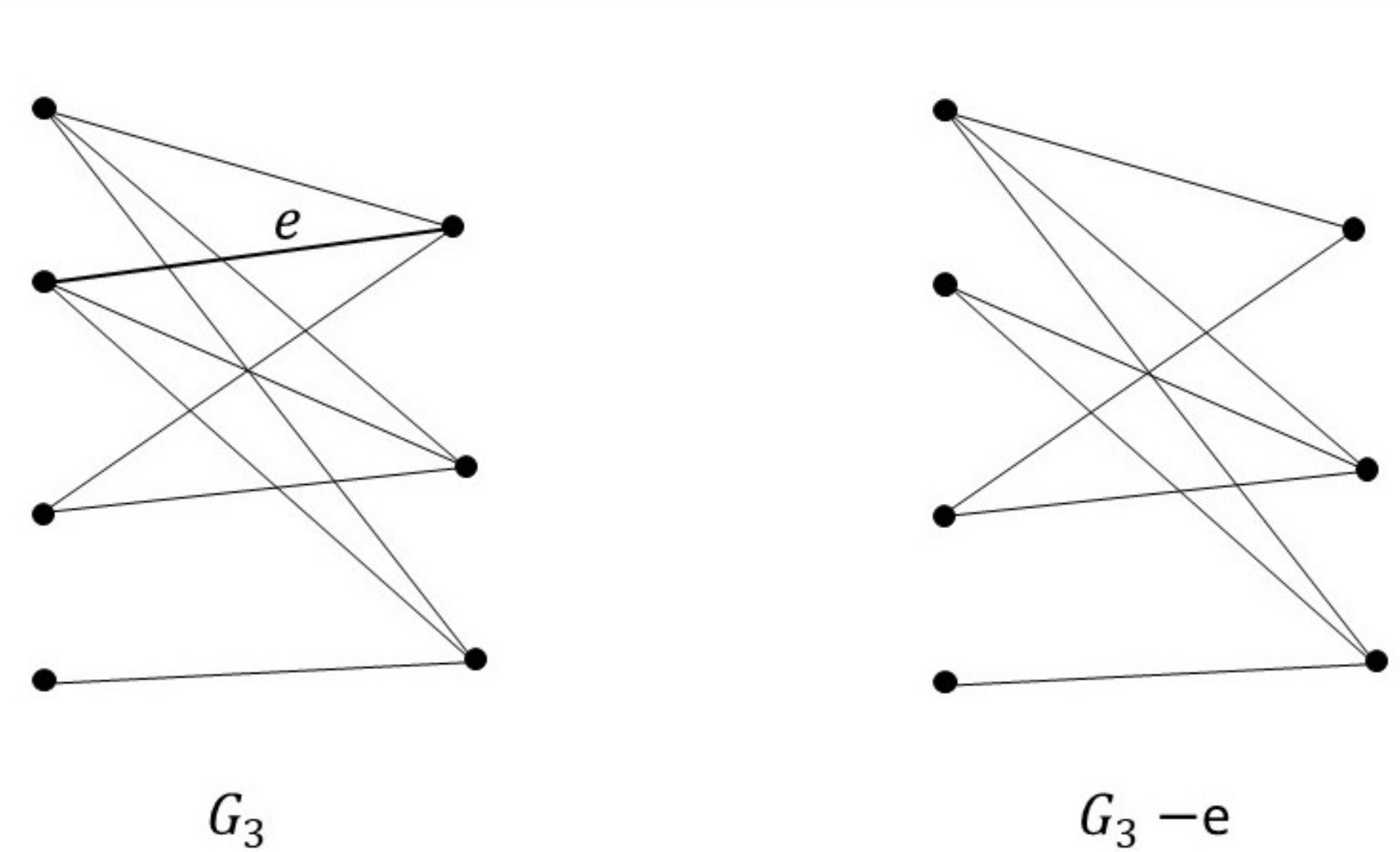}
	\caption{$G_3$ and $G_3-e$}
	\label{fig3}
	\includegraphics[scale= 0.40]{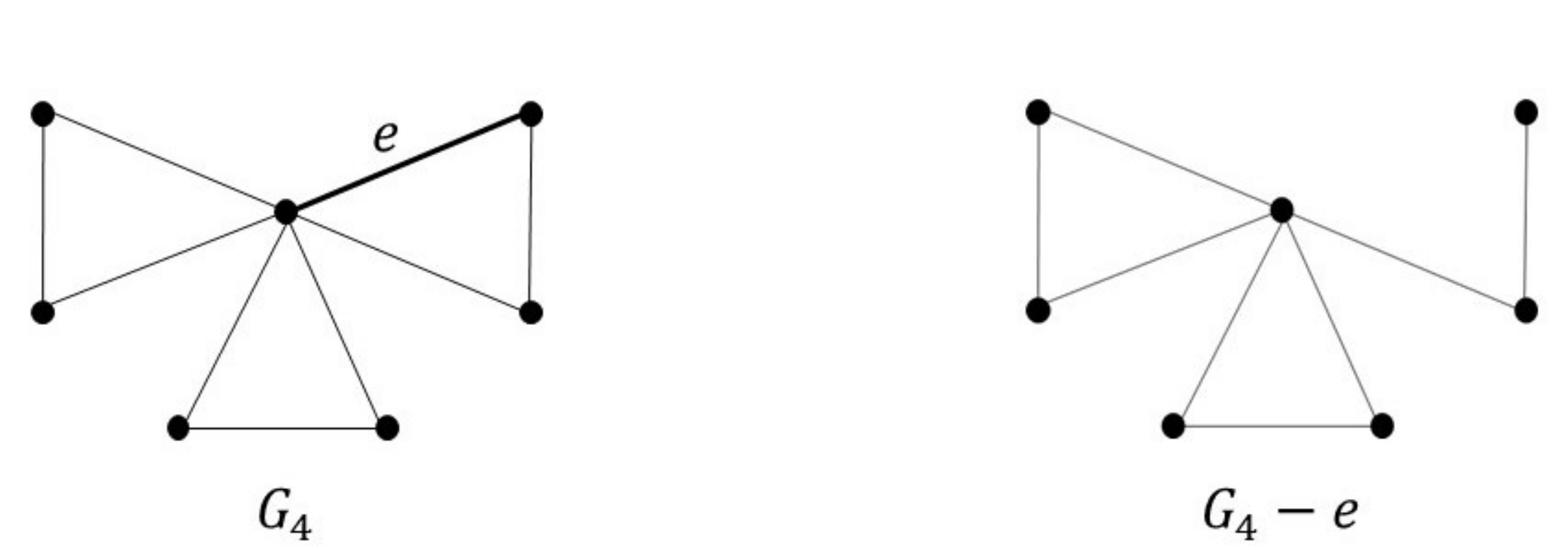}
	\caption{$G_4$ and $G_4-e$}
	\label{fig4}
\end{figure}
\begin{table}[h!]
	\centering
	\begin{tabular}{|p{1.2cm}|p{1.8cm}|p{1.8cm}|p{1.2cm}|p{1.8cm}|p{1.8cm}|p{2.1cm}|p{2.1cm}|}
		\hline
		Graphs & D-energy & $\varepsilon$-energy & Graphs &  D-energy & $\varepsilon$-energy & D-energy comparison & $\varepsilon$-energy comparison\\
		\hline
		$G_3$ &   $\approx 23.5415 $ & $\approx 22.6856$ &$G_3-e$ &  $\approx 23.4115 $ & $\approx 24.0922$ & $E_D(G_3)>E_D(G_3-e)$ & $E_{\varepsilon}(G_3)< E_{\varepsilon}(G_3-e)$ \\
		\hline
		$G_4$ & $\approx 19.2470$ & $\approx 17.3808$ & $G_4-e$  & $\approx 22.4508$ & $\approx 16.9706$ &$E_D(G_4)<E_D(G_4-e)$ &$E_{\varepsilon}(G_4)> E_{\varepsilon}(G_4-e)$ \\
		\hline
	\end{tabular}
	\caption{Comparison of distance energy and eccentricity energy change}
	\label{tab2}
\end{table}

In this article, we study the eccentricity energy change of complete $k$-partite graphs, and prove that the eccentricity energy of $K_{n_1,\hdots,n_k}$ with $k\geq 2$ and $ n_i\geq 2$ always increases due to  an edge deletion.

\section{Eccentricity energy change of complete $k$-partite graphs}\label{sec3}
%In this section, we consider the complete $k$-partite graph $K_{n_1,n_2,\hdots,n_k}$ with $n_i\geq 2$ and $k\geq 2$. Since $n_i\geq 2$, both $K_{n_1,n_2,\hdots,n_k}$ and $K_{n_1,n_2,\hdots,n_k}-e$ are connected for any edge $e$.

 It is known that the $\varepsilon$-eigenvalues of $K_{n_1,n_2,\hdots,n_k}$ are $-2$ with multiplicity $(n_1+n_2+\hdots+n_k-k)$ and $2(n_i-1)$ for $i=1,2,\hdots,k$ \cite{mahato2020}. Thus $E_{\varepsilon}(K_{n_1,n_2,\hdots,n_k})=4(n_1+n_2+\hdots+n_k-k)$. But it is not easy to compute the $\varepsilon$-eigenvalues and the $\varepsilon$-energy of $K_{n_1,n_2,\hdots,n_k}-e$ explicitly. First, in Theorem \ref{main2}, we show that the  $\varepsilon$-energy of complete bipartite graphs increase due to an edge deletion. Using this result, we shall prove that $E_{\varepsilon}(K_{n_1,n_2,\hdots,n_k})<E_{\varepsilon}(K_{n_1,n_2,\hdots,n_k}-e)$ for any edge $e$.

To begin with, in the next lemma, we establish that $-2$ is an $\varepsilon$-eigenvalue of $K_{m,n}-e$ with multiplicity at least $(m+n-4)$ by constructing the corresponding  eigenvectors explicitly. Indeed, later we will see that the multiplicity of $-2$ equals to $(m+n-4).$
\begin{lem}\label{lem1}
Let $K_{m,n}$ be the complete bipartite graph with $m, n \geq 2$. Then for any edge $e$, $-2$ is an $\varepsilon$-eigenvalue of $K_{m,n}-e$ with multiplicity at least $m+n-4$. 
\end{lem}
\begin{proof}
Let $V_1$ and $V_2$ be the  bipartition of the vertex set of $K_{m,n}$ with $V_1=\{v_1,\hdots, v_m\}$ and $V_2=\{v_{m+1},v_{m+2}, \hdots, v_{m+n}\}$. Let $e$ be any edge in $K_{m,n}$. Without loss of generality, we assume that $e=v_1v_{m+1}$. Let $e_k$ denote the $(m+n)$-vector with $k$-th entry $1$ and $0$ elsewhere. Let $X=\{e_2-e_i:i=3,4,\hdots,m\}\cup \{e_{m+2}-e_j:j=m+3,m+4,\hdots,m+n\}$. Then for any vector $x \in X$, we have $\varepsilon(K_{m,n}-e)x=-2x$. Since $|X|=m+n-4$, and all vectors in $X$ are linearly independent, $\varepsilon(K_{m,n}-e)$ has $-2$ as an eigenvalue  with multiplicity at least $m+n-4$.
\end{proof}

To find the remaining $\varepsilon$-eigenvalues of $K_{m,n}-e$, we  use the equitable partition technique.

\begin{dfn}[\cite{brouwer2011spectra}]
	(Equitable partition) Let $A$ be a real symmetric matrix whose rows and columns are indexed by $X=\{1,2,\hdots,n\}$. Let $\pi=\{X_1,X_2,\hdots,X_m\}$ be a partition of $X$. The characteristic matrix $C$ is an $n\times m$ matrix whose $j$-th column is the characteristic vector of the set $X_j$ $(j=1,2,\hdots,m)$. Let $A$ be partitioned conformally with $\pi$ as follows \[A=\left[ {\begin{array}{cccc}
			A_{11} & A_{12} &\hdots & A_{1m}\\
			A_{21} & A_{22} &\hdots & A_{2m}\\
			\vdots &\hdots & \ddots & \vdots\\
			A_{m1} & A_{m2}& \hdots &A_{mm}\\
	\end{array} } \right],\]
	where $A_{ij}$ denotes the submatrix (block) of $A$ formed by rows in $X_i$ and the columns in $X_j$. If $q_{ij}$ denote the average row sum of $A_{ij}$, then the matrix $Q=(q_{i,j})$ is called the quotient matrix of $A$. If the row sum of each block $A_{ij}$ is  constant, then the partition $\pi$ is called equitable partition.
\end{dfn}
Now we state a well-known result  about the spectrum of a quotient matrix corresponding to an equitable partition.

\begin{thm}[\cite{brouwer2011spectra}\label{quo-spec}]
	Let $Q$ be a quotient matrix of any square matrix $A$ corresponding to an equitable partition. Then the spectrum of $A$ contains the spectrum of $Q$.
\end{thm}

Let us partition the vertex set of $K_{m,n}-e$ as follows: $\pi = V(K_{m,n}-e) = P_1 \cup P_2 \cup P_3 \cup P_4$, where $P_1 = \{v_1\}, P_2 = \{v_2, \dots, v_m\}, P_3 = \{u_1\}$ and $P_4 = \{u_2, \dots, u_n\}$. Note that $\pi$ is an equitable partition of $\varepsilon(K_{m,n}-e)$ and the corresponding quotient matrix $Q_\pi$ is given by
\[Q_\pi=\begin{bmatrix}
0 & 2(m-1) & 3 & 0\\
2 & 2(m-2) & 0 & 0 \\
3 & 0 & 0 & 2(n-1) \\
0 & 0 & 2 & 2(n-2)
\end{bmatrix}.\]
The characteristic polynomial of $Q_\pi$ is 
\begin{equation}\label{eqn1}
p(x) = x^4-2(m+n-4)x^3 + [4mn-12(m+n) + 15]x^2 + [16mn - 6(m+n) -40]x - 4[5mn-14(m+n)+32].
\end{equation}
The eigenvalues of $Q_\pi$ are the roots of the characteristic polynomial $p(x)$. Since $Q_\pi$ is the quotient matrix corresponding to the equitable partition $\pi$, by Lemma \ref{quo-spec}, we have every eigenvalue of $Q_\pi$ is an eigenvalue of $\varepsilon(K_{m,n}-e)$. Therefore, the roots of the polynomial $p(x)$ are in the  $\varepsilon$-spectrum of $K_{m,n}-e$. In the following lemmas, we analyze the roots of $p(x)$, and compare the eccentricity energies of $K_{m,n}$ and $K_{m,n}-e$ for  $m = 2, 3$.

\begin{lem}\label{lem2}
Let $n\geq 2$. If $e$ is any edge of $K_{2,n}$, then 
$E_{\varepsilon}(K_{2,n})<E_{\varepsilon}(K_{2,n}-e).$
\end{lem}
\begin{proof}
It is easy to see that $E_{\varepsilon}(K_{2,2})=8<10=E_{\varepsilon}(K_{2,2}-e)$ and   $E_{\varepsilon}(K_{2,3})=12<13.8486=E_{\varepsilon}(K_{2,3}-e)$.

Let $n\geq 4$ and  $\alpha_1\geq \alpha_2 \geq \alpha_3\geq \alpha_4$ be the roots of the equation
\begin{equation}\label{eqnfor2}
p(x)=x^4+(4-2n)x^3-(4n+9)x^2+26(n-2)x+16(n-1) =0.
\end{equation}
Since $\alpha_1, \alpha_2, \alpha_3, \alpha_4$ are the roots of the equation (\ref{eqnfor2}), we have
\begin{eqnarray*}
\alpha_1+\alpha_2+\alpha_3+\alpha_4 &=& 2n-4\\
\alpha_1 \alpha_2+\alpha_1\alpha_3+\alpha_1 \alpha_4+ \alpha_2\alpha_3+\alpha_2 \alpha_4+\alpha_3\alpha_4 &=& -(4n+9)\\
\alpha_1 \alpha_2\alpha_3+\alpha_1 \alpha_2\alpha_4+ \alpha_1\alpha_3 \alpha_4+\alpha_2\alpha_3\alpha_4 &=& -26(n-2)\\
\alpha_1 \alpha_2 \alpha_3 \alpha_4 &=& 16(n-1).
\end{eqnarray*}

Since $n\geq 4$, $\alpha_1 \alpha_2\alpha_3\alpha_4=16(n-1)>0$ and hence $p(x)$ has either $0,2$ or $4$ positive roots. As $n \geq 4$, $\alpha_1+\alpha_2+\alpha_3+\alpha_4 = 2n-4>0$, so all $\alpha_i$'s can not be negative. Again $\alpha_1 \alpha_2+\alpha_1\alpha_3+\alpha_1 \alpha_4+ \alpha_2\alpha_3+\alpha_2 \alpha_4+\alpha_3\alpha_4 =-(4n+9)<0$ implies that $\alpha_i$'s cannot be all positive. Thus $p(x)$ must have exactly two positive and two negative roots. Assume that $\alpha_4 \leq  \alpha_3<0<\alpha_2 \leq \alpha_1$.

Now $p(-5)=36n+144>0$ and $p(-4)=-24(n-2)<0$.  By the intermediate value theorem, we have $-5<\alpha_4<-4$. As $p(-1)=-12(n-2)<0$ and $p(0)=16(n-1)>0$, we have $-1<\alpha_3<0$. Consequently,
$$\alpha_3+\alpha_4<(-4)+0=-4.$$

Now it is clear that all $\alpha_i$s are different from $-2$,   so $E_{\varepsilon}(K_{2,n}-e)=2n-4+\abs{\alpha_1}+\abs{\alpha_2}+\abs{\alpha_3}+\abs{\alpha_4}$. Note that $E_{\varepsilon}(K_{2,n})=4(2+n-2)=4n$.

Therefore,
\begin{eqnarray*}
E_{\varepsilon}(K_{2,n}-e) &=&2n-4+\abs{\alpha_1}+\abs{\alpha_2}+\abs{\alpha_3}+\abs{\alpha_4}\\
&=& 2n-4+2n-4-2(\alpha_3+\alpha_4)\\
&=& 4n-8-2(\alpha_3+\alpha_4)\\
&>& 4n=E_{\varepsilon}(K_{2,n}).
\end{eqnarray*}
This completes the proof.
\end{proof}

\begin{lem}\label{lem3}
Let  $n\geq 3$. If $e$ is any edge  in $K_{3,n}$, then
$E_{\varepsilon}(K_{3,n})<E_{\varepsilon}(K_{3,n}-e).$
\end{lem}
\begin{proof}
Note that $E_{\varepsilon}(K_{3,n})=4(3+n-2)=4n+4$.  Next, we show that  $E_{\varepsilon}(K_{3,n}-e)=2n-2+\abs{\alpha_1}+\abs{\alpha_2}+\abs{\alpha_3}+\abs{\alpha_4}$, where $\alpha_1\geq \alpha_2 \geq \alpha_3\geq \alpha_4$ are eigenvalues of the quotient matrix, with $m=3$, corresponding to the equitable partition discussed before Lemma \ref{lem2}.  The characteristic equation of the quotient matrix $Q_\pi$ simplifies to
\begin{equation}\label{eqnfor3}
p(x)=x^4+(2-2n)x^3-21x^2+(42n-58)x+4(10-n) =0.
\end{equation}
Since $\alpha_1, \alpha_2, \alpha_3, \alpha_4$ are the roots of the equation (\ref{eqnfor3}), we have
\begin{equation}\label{eqnfor3,1}
\alpha_1+\alpha_2+\alpha_3+\alpha_4=2n-2
\end{equation}
\begin{equation}
\alpha_1 \alpha_2+\alpha_1\alpha_3+\alpha_1 \alpha_4+ \alpha_2\alpha_3+\alpha_2 \alpha_4+\alpha_3\alpha_4= -21
\end{equation}
\begin{equation}\label{eqnfor42}
\alpha_1 \alpha_2\alpha_3+\alpha_1 \alpha_2\alpha_4+ \alpha_1\alpha_3 \alpha_4+\alpha_2\alpha_3\alpha_4=-(42n-58)
\end{equation}
\begin{equation}\label{eqnfor3,4}
\alpha_1 \alpha_2\alpha_3\alpha_4=4(10-n).
\end{equation}

\textbf{Case 1}: Let $n<10$. Then $\alpha_1 \alpha_2\alpha_3\alpha_4=4(10-n)>0$, and hence $p(x)$ has either $0,2$ or $4$ positive roots. As $n \geq 3$, $\alpha_1+\alpha_2+\alpha_3+\alpha_4 = 2n-2>0$ and so all $\alpha_i$'s cannot be negative. Again $\alpha_1 \alpha_2+\alpha_1\alpha_3+\alpha_1 \alpha_4+ \alpha_2\alpha_3+\alpha_2 \alpha_4+\alpha_3\alpha_4 =-21<0$ implies that $\alpha_i$'s cannot be all positive. Thus $p(x)$ must have exactly two positive and two negative roots. Assume that $\alpha_4 \leq  \alpha_3<0<\alpha_2 \leq \alpha_1$.

Now $p(-5)=36(n+5)>0$ and $p(-4)=-4(11n-16)<0$. Thus $-5<\alpha_4<-4$. As $p(-1)=-4(11n-19)<0$ and $p(0)=4(10-n)>0$, we have $-1<\alpha_3<0$. Consequently,
$$\alpha_3+\alpha_4<(-4)+0=-4.$$
Thus all the roots of (\ref{eqnfor3}) are different from $-2$, and hence we have the following: 
\begin{eqnarray*}
E_{\varepsilon}(K_{3,n}-e) &=&2n-2+\abs{\alpha_1}+\abs{\alpha_2}+\abs{\alpha_3}+\abs{\alpha_4}\\
&=& 2n-2+2n-2-2(\alpha_3+\alpha_4)\\
&=& 4n-4-2(\alpha_3+\alpha_4)\\
&>& 4n+4=E_{\varepsilon}(K_{3,n}).
\end{eqnarray*}

\textbf{Case 2}: Let $n>10$. Then $\alpha_1 \alpha_2\alpha_3\alpha_4=4(10-n)<0$,  and hence $p(x)$ has either one or three negative roots.  Suppose that $p(x)$ has three negative roots, say $\alpha_2, \alpha_3$ and $\alpha_4$. From (\ref{eqnfor3,1}), we have
\begin{align*}
 \alpha_1+\alpha_2=2(n-1)-\alpha_3-\alpha_4>0, \qquad as \quad \alpha_3,\alpha_4<0.
\end{align*}

Since $\alpha_2, \alpha_3, \alpha_4<0$ and $\alpha_1+\alpha_2>0$, therefore $\alpha_1 \alpha_2\alpha_3+\alpha_1 \alpha_2\alpha_4+ \alpha_1\alpha_3 \alpha_4+\alpha_2\alpha_3\alpha_4=\alpha_1 \alpha_2\alpha_3+\alpha_1 \alpha_2\alpha_4+ (\alpha_1+\alpha_2)\alpha_3\alpha_4 >0$. But from (\ref{eqnfor42}), we have $\alpha_1 \alpha_2\alpha_3+\alpha_1 \alpha_2\alpha_4+ \alpha_1\alpha_3 \alpha_4+\alpha_2\alpha_3\alpha_4=-(42n-58)<0$. Hence $p(x)$ has exactly one negative root:
$$\alpha_4<0<\alpha_3 \leq \alpha_2\leq \alpha_1.$$

Now $p(-5)=36(n+5)>0$ and $p(-4)=-4(11n-16)<0$ and hence $-5<\alpha_4<-4$. 
Thus all the roots of (\ref{eqnfor3}) are different from $-2$, and hence we have the following: 
\begin{eqnarray*}
E_{\varepsilon}(K_{3,n}-e) &=&2n-2+\abs{\alpha_1}+\abs{\alpha_2}+\abs{\alpha_3}+\abs{\alpha_4}\\
&=& 2n-2+2n-2-2\alpha_4\\
&=& 4n-4-2\alpha_4\\
&>& 4n+4=E_{\varepsilon}(K_{3,n}).
 \end{eqnarray*}

\textbf{Case 3:} For $n=10$, we have $E_{\varepsilon}(K_{3,10})=44<45.0087=E_{\varepsilon}(K_{3,10}-e)$. Thus for $n\geq 3$ and any edge $e$, we have
$$E_{\varepsilon}(K_{3,n})<E_{\varepsilon}(K_{3,n}-e).$$
\end{proof}

In the next crucial lemma, we show that the  quotient matrix $Q_\pi$ has exactly one negative eigenvalue for $n \geq 5$ and $m \geq 4$.
\begin{lem}\label{lem4}
If $m\geq 4$ and $n\geq 5$, then $p(x)$ (defined in (\ref{eqn1})) has exactly three positive and one negative roots, that is, $\alpha_4<0<\alpha_3 \leq \alpha_2\leq \alpha_1$. Moreover, $-5<\alpha_4<-4$.
\end{lem}
\begin{proof}
Let $\alpha_1, \alpha_2, \alpha_3, \alpha_4$ be the roots of the polynomial
\begin{eqnarray*}
p(x) &=& x^4 -2(m+n-4)x^3+(4mn-12(m+n)+15)x^2 +(16mn-6(m+n)-40)x\\
 & & - 4(5mn-14(m+n)+32).
\end{eqnarray*}
Then
\begin{equation}\label{eqnform,1}
 \alpha_1+\alpha_2+\alpha_3+\alpha_4=2(m+n-4)
\end{equation}
\begin{equation}
\alpha_1 \alpha_2+\alpha_1\alpha_3+\alpha_1 \alpha_4+ \alpha_2\alpha_3+\alpha_2 \alpha_4+\alpha_3\alpha_4=4mn-12(m+n)+15
\end{equation}
\begin{equation}
\alpha_1 \alpha_2\alpha_3+\alpha_1 \alpha_2\alpha_4+ \alpha_1\alpha_3 \alpha_4+\alpha_2\alpha_3\alpha_4=-(16mn-6(m+n)-40)
\end{equation}
\begin{equation}\label{eqnform,4}
\alpha_1 \alpha_2\alpha_3\alpha_4=-4(5mn-14(m+n)+32).
\end{equation}

Now $\alpha_1 \alpha_2 \alpha_3 \alpha_4 = - 4[5mn-14(m+n)+32]=-4[9(m-2)(n-2)-4(m-1)(n-1)]<0$. Thus $p(x)$ has either one or three negative roots. Suppose that $p(x)$ has three negative roots, say $\alpha_2, \alpha_3$ and $\alpha_4$. From (\ref{eqnform,1}), we have
\begin{align*}
 \alpha_1+\alpha_2=2(m+n-4)-\alpha_3-\alpha_4>0, \qquad as \quad \alpha_3,\alpha_4<0.
\end{align*}

Since $\alpha_2, \alpha_3, \alpha_4<0$ and $\alpha_1+\alpha_2>0$, we have $\alpha_1 \alpha_2\alpha_3+\alpha_1 \alpha_2\alpha_4+ \alpha_1\alpha_3 \alpha_4+\alpha_2\alpha_3\alpha_4=\alpha_1 \alpha_2\alpha_3+\alpha_1 \alpha_2\alpha_4+ (\alpha_1+\alpha_2)\alpha_3\alpha_4 >0$. But $\alpha_1 \alpha_2\alpha_3+\alpha_1 \alpha_2\alpha_4+ \alpha_1\alpha_3 \alpha_4+\alpha_2\alpha_3\alpha_4=-(16mn-6(m+n)-40)=-2[(m-4)(n-5)+7n(m-1)+2m]<0$. Hence $p(x)$ has exactly one negative root.
$$\alpha_4<0<\alpha_3 \leq \alpha_2\leq \alpha_1.$$

Moreover, $p(-5)=36(m+n)+72>0$ and $p(-4)=-4[5mn-4(m+n)-4]=-4[4n(m-1)+m(n-4)-4]<0$, and hence $-5<\alpha_4<-4$.
\end{proof}

Using the previous lemmas, next we show that the  $-2$ as an $\varepsilon$-eigenvalue of $K_{m,n}-e$ has multiplicity $m+n-4$.
\begin{thm}\label{mainthm}
Let $K_{m,n}$ be a complete bipartite graph with $m, n \geq 2$. Then for any edge $e$, the $\varepsilon$-eigenvalues of $K_{m,n}-e$ are $-2$ with multiplicity $m+n-4$, and the roots of the polynomial
$$p(x) = x^4-2(m+n-4)x^3 + [4mn-12(m+n) + 15]x^2 + [16mn - 6(m+n) -40]x - 4[5mn-14(m+n)+32].$$
\end{thm}
\begin{proof}
From Lemma \ref{lem2}, Lemma \ref{lem3} and Lemma \ref{lem4}, we have $-2$ can not be a root of the polynomial $p(x)$, for all $m,n \geq 2$. Since the roots of the polynomial $p(x)$ are in the $\varepsilon$-spectrum of $K_{m,n}-e$, and all the roots of $p(x)$ are different from $-2$, therefore the roots of  $p(x)$ are the remaining $\varepsilon$-eigenvalues of $K_{m,n}-e$. 
\end{proof}

In the next theorem, we establish that the eccentricity energy of a complete bipartite graph $K_{m,n}$ always increases due to  deletion of any edge.

\begin{thm}\label{main2}
Let $K_{m,n}$ be the complete bipartite graph with $m,n\geq 2$. Then for any edge $e$ of $K_{m,n}$,
$$E_{\varepsilon}(K_{m,n})<E_{\varepsilon}(K_{m,n}-e).$$
\end{thm}
\begin{proof}
Without loss of generality, let us assume that $n\geq m \geq 2$. Let us consider the following cases:

\noindent\textbf{Case 1:} Let $m=2$. For $n\geq 2$, by Lemma \ref{lem2}, we have
$$E_{\varepsilon}(K_{2,n})<E_{\varepsilon}(K_{2,n}-e).$$

\noindent\textbf{Case 2:} Let $m=3$. For $n\geq $3, by Lemma \ref{lem3}, we have
$$E_{\varepsilon}(K_{3,n})<E_{\varepsilon}(K_{3,n}-e).$$

\noindent\textbf{Case 3:} Let $m\geq 4$. For $m=n=4$, by direct calculation we have
$$E_{\varepsilon}(K_{4,4})=24<24.84886=E_{\varepsilon}(K_{4,4}-e).$$
Let  $n\geq 5$. By Theorem \ref{mainthm}, the $\varepsilon$-eigenvalues of $K_{m,n}-e$ are $-2$ with multiplicity $m+n-4$, and the roots $\alpha_4\leq \alpha_3 \leq \alpha_2\leq \alpha_1$ of the polynomial $$p(x) = x^4-2(m+n-4)x^3 + [4mn-12(m+n) + 15]x^2 + [16mn - 6(m+n) -40]x - 4[5mn-14(m+n)+32].$$
Since $m\geq 4, n\geq 5$, by Lemma \ref{lem4}, we have $\alpha_4<-4$. Therefore,
\begin{eqnarray*}
E_{\varepsilon}(K_{m,n}-e) &=&2(m+n-4)+\abs{\alpha_1}+\abs{\alpha_2}+\abs{\alpha_3}+\abs{\alpha_4}\\
&=&2(m+n-4)+2(m+n-4)-2\alpha_4\\
&=& 4m+4n-16-2\alpha_4\\
&>& 4m+4n-8=E_{\varepsilon}(K_{m,n}).
\end{eqnarray*}
This completes the proof.
\end{proof}
Next, using Theorem \ref{main2}, we prove that the eccentricity energy of a complete $k$-partite graph increases when an edge deleted.

\begin{thm}\label{main3}
Let $K_{n_1,n_2,\hdots,n_k}$ be the complete $k$-partite graph such that $\sum_{i=1}^k n_i=n$ and $n_i\geq 2$ for $i=1,2,\hdots,k$. Then for any edge $e$ of $K_{n_1,n_2,\hdots,n_k}$,
$$E_{\varepsilon}(K_{n_1,n_2,\hdots,n_k})<E_{\varepsilon}(K_{n_1,n_2,\hdots,n_k}-e).$$
\end{thm}
\begin{proof}
Let $V_1,V_2,\hdots,V_k$ be the $k$ partition of  the vertex set of $K_{n_1,n_2,\hdots,n_k}$, where
$V_1=\{v_1,v_2,\hdots,v_{n_1}\}$ and $V_i=\Big \{v_{1+\sum_{j=1}^{i-1}n_j},v_{2+\sum_{j=1}^{i-1}n_j},\hdots,v_{\sum_{j=1}^{i}n_j}\Big \}$ for $i=2,3,\hdots,k$. Let $e$ be any edge in $K_{n_1,\hdots,n_k}$. Without loss of generality, we assume that $e=v_1v_{n+1}$. Then,
\begin{align*}
\varepsilon(K_{n_1,\hdots,n_k})=
	\left[ {\begin{array}{cc}
	\varepsilon(K_{n_1,n_2}) & 0\\
	0 & B\\
	\end{array} } \right]\quad and \quad
\varepsilon(K_{n_1,\hdots,n_k}-e)=
	\left[ {\begin{array}{cc}
	\varepsilon(K_{n_1,n_2}-e) & 0\\
	0 & B\\
	\end{array} } \right],
\end{align*}
where		
	\[B=\left[ {\begin{array}{cccc}
2(J_{n_3}-I_{n_3}) & 0 & \hdots  & 0 \\
0 & 2(J_{n_4}-I_{n_4}) & \hdots  & 0 \\
\vdots & \vdots & \ddots & \vdots \\
0 & 0& \hdots & 2(J_{n_k}-I_{n_k}) \\
\end{array} } \right].\]
By using Theorem \ref{main2}, we have
\begin{eqnarray*}
E_{\varepsilon}(K_{n_1,\hdots,n_k})&=& E_{\varepsilon}(K_{n_1,n_2})+E_{\varepsilon}(B)\\
&<& E_{\varepsilon}(K_{n_1,n_2}-e)+E_{\varepsilon}(B)\\
&=& E_{\varepsilon}(K_{n_1,\hdots,n_k}-e).
\end{eqnarray*}
\end{proof}

\textbf{Acknowledgement:} We would like to thank  Department of Science and Technology, India, for financial support through the projects Early Carrier Research Award (ECR/2017/000643) and MATRICS (MTR/2018/000986).

\bibliographystyle{amsplain}
\bibliography{ecc-ref}

\end{document}